\documentclass[12pt]{amsart}
\usepackage{tikz,amsfonts,amssymb,stmaryrd,amscd,amsmath,latexsym,amsbsy,bbold}
\usepackage{graphicx}
\usepackage{amssymb}
\usepackage{amsfonts}
\usepackage{cite}
\usepackage{latexsym}
\usepackage{pdfsync}
\usepackage{mathrsfs}
\usepackage{mathtools}
\usepackage{amsthm}
\usepackage{indentfirst}
\usepackage{comment}
\usepackage[open,openlevel=2]{bookmark}

\newtheorem{defn}{Definition}
\newtheorem{lmm}{Lemma}
\newtheorem*{thm*}{Theorem}
\newtheorem{prop}{Proposition}
\newtheorem*{rmk*}{Remark}
\newtheorem*{coro*}{Corollary}

\newcommand{\R}{\ensuremath{\mathbb{R}}}
\newcommand{\Z}{\ensuremath{\mathbb{Z}}}
\newcommand{\N}{\ensuremath{\mathbb{N}}}

\newcommand{\Id}{\ensuremath{\mathrm{Id}}}

\newcommand{\la}{\ensuremath{\langle}}
\newcommand{\ra}{\ensuremath{\rangle}}
\newcommand{\mk}{\ensuremath{\mathfrak}}

\begin{document}
\title{2D Inverse Problem with a Foliation Condition}

\author{Qiuye Jia}

\maketitle

\begin{abstract}
We consider the geodesic X-ray transform in two dimension under the assumption that the boundary is convex and the region has a foliation structure. For functions that are constant on each layer of the foliation, we prove invertibility and a stability estimate of the geodesic X-ray transform. 
\end{abstract}
\tableofcontents

\section{Introduction}
\label{intro}
The geodesic X-ray transform is a generalization of the Radon transform, 
and the inverse problem on it 
can be formulated as follows: 
On a Riemannian manifold $(X,g)$, the information we have are integrals like $(I_\varrho f)(\gamma(\cdot)) := \int_\gamma \varrho  f(\gamma(t))dt$, where 
$\gamma$ is a geodesic segment in a neighborhood $O_p$ of a fixed point $p \in \partial X$, 
and $\varrho$ is a density function on $T^*X$, the cotangent bundle of $X$.

In this paper, we consider the local geodesic ray transform with weights in dimension 2.
`Local' means that 
the geodesic segment we integrate over lies in $O_p$ and has endpoints on $\partial X$, 
see Section \ref{sec_defn_transform} for the more detailed definition.
We show that we can recover the restriction of the function $f$ to $O_p$,
which amounts to the injectivity of $I_\varrho$,
by proving an estimate which uses the Sobolev norm of $I_\varrho f$ (viewed as a function on the 
projective sphere bundle $PSX$) to control the Sobolev norm of $f$. 
In addition,
this estimate is stable under small perturbations of the metric $g$.

This problem is resolved by Uhlmann and Vasy  \cite{uhlmann2016inverse} in dimension $\geq 3$ under convexity assumption.
Using the same framework, Paternain et al. \cite{paternain2019geodesic} extended this result to the case with matrix weights. 
The injectivity of the global Radon transform with 
positive real analytic weights is shown by Boman and Quinto \cite{boman1993support},
which implies the local injectivity since their result does not require smoothness of 
the transformed function and we can extend by 0 and apply the global result.

For the two dimensional case, Pestov and Uhlmann \cite{pestov2005two} showed boundary rigidity in the two dimensional compact simple case.
In fact, they showed that the scattering relation determines the Dirichlet-to-Neumann map. By this and the result in \cite{lassas2001determining}, we know two metrics in the boundary rigidity problem can only differ by a conformal factor. Combining this with the result of Mukhometov \cite{mukhometov1981problem}, 
which generalized the result in \cite{mukhometov19772Dreconstruction}, we know that this conformal factor is 1, proving the boundary rigidity result.

Boman proved \cite{boman1993example}\cite{boman2011local} that without the assumption that the weights are analytic,
the Radon transform is not locally injective in the 2-dimensional case 
by constructing non-zero functions and weights with vanishing Radon transform. 
On the other hand, the injectivity is expected to hold under certain restrictions on the function and the geometric structure. Here the additional condition we impose is the convexity of $\partial X$ 
as in \cite{uhlmann2016inverse} and 
the function is adapted to it in the sense we define below. Roughly speaking, this means that the direction that the function changes is conormal to the layer structure of the manifold. In the real world application, this can be interpreted as the situation where the data is sensitive to depth but not the position along layers.
For more results on injectivity of geodesic ray transforms and
stability estimates, see \cite[Chapter~4]{paternain2022geometric}.

We give an intuitive explanation on the difference between the 2-dimensional case and the higher dimensional case.
As shown below in the discussion before (\ref{symbol_main_term}), the leading contribution 
to the principal symbol is given by directions othogonal to 
the covector we are evaluating the principal symbol.
However, we can only choose the integrand $\chi \varrho$
 to be positive on certain fixed directions, or more concretely, on 
directions almost tangential to our surfaces of the foliation.
In the two dimensional case, we only have one such direction
(two if you consider opposite directions as different),
however we can not guarantee that this part carrying the 
positivity of the integrant is the direction othogonal 
to the covector variable of the symbol, which should give the 
leading term of the symbol.
From a microlocal perspective, this corresponds to the fact that the
ellipticity on the entire cotangent bundle fails. However, the ellipticity still holds if we restrict the directions on the fibre part of the cotangent bundle. So we identify the directions on which our operator behaves well, and modify the symbol on other directions to obtain complete ellipticity.

We need to emphasize that, this failure of ellipticity only happens at the fiber infinity on the `normal' direction of the scattering cotangent bundle. While on other directions, and when
fiber variables are bounded, this principal symbol is still elliptic.
See (\ref{symbol_boundary}) and the discussion after it.

In the next section, we introduce notations, the definition of the geodesic ray transform and state the main result. In Section \ref{calculus} we recall some basic facts about Sobolev spaces and scattering calculus. We describe and prove important properties of the conjugated normal operator of the geodesic ray transform in Section \ref{osc}. Finally in Section \ref{main_proof} we prove the main theorem.

\section{Notations and results}
\label{notations}
\subsection{General notations}
Let $(X,g)$ be a two dimensional Riemannian manifold with boundary. 
It is convenient to consider a larger region containing $X$. So suppose $X$ is embedded as a strictly convex domain in a Riemannian manifold $(\tilde{X},g)$ (we have used the same notation to indicate the smooth extension of the metric). Here convexity means when a geodesic is tangent to $\partial X$, it is tangent and curving away from $X$. 

Concretely, let $\bar{X}$ be the closure of $X$ in $\tilde{X}$ and let
$\rho$ be the boundary defining function of $\tilde{X}$, which means $\rho(z)$ vanishes on $\partial X$, $\rho(z)>0$ on $X$, and satisfy the non-degeneracy condition $d\rho \neq 0$ when $\rho=0$.
Using $G$ to denote the dual metric function on $T^*\tilde{X}$,
the convexity means that if at some $\beta \in T_p^*\tilde{X}\backslash o$ with $p \in \partial X$ and $o$ being the zero section, we have: 
\begin{equation} \label{convexity}
\text{if }(H_G \rho) (\beta)= 0 \text{, then }(H^2_G \rho) (\beta) < 0.
\end{equation}

We will consider local geodesic transform near $p$ in a neighborhood $O_p \subset U$ of $p$ in $X$. See Section \ref{sec_defn_transform} for more details on $O_p$ and the meaning of local here. 
Recall that an initial point and an tangent vector at this point determine a geodesic.
The bundle we use to parametrize geodesics is the \textit{projective sphere bundle}, denoted by $PSX$, whose fibers are $\R \times \mathbb{S}^{n-2}$, which is $\R \times \{\pm 1\}$ in our case. Each fiber has two components. 
It parametrizes geodesics whose initial velocities has unit tangential component, except for those ones that are normal to our foliation (corresponding to $\lambda=\pm \infty)$. Those excluded geodesics are irrelevant for our purpose since our cut-off $\chi$
is restricting our analysis to those geodesics that are almost tangential to our foliation.

\subsection{The foliation condition and the choice of the coordinate system}
We first introduce another boundary defining function $\tilde{x}$ satisfying 
\begin{equation} \label{defn_tildex}
d\tilde{x}(p)=-d\rho(p), \, \tilde{x}(p)=0,
\end{equation} 
whose level sets are strictly convex from the sublevel sets $\{\tilde{x}<-T\}$ for a constant $T>0$,
 which means geodesics tangential to this region will curve away from it. 
This function is used to introduce the artificial boundary, 
which is a level set of $\tilde{x}: \, \{\tilde{x}=-c\}$ for $c>0$.
This level set intersects with $\partial X$ and 
together with $\partial X$ it 
encloses a small region on which our discussion happens.
This allows us to conduct analysis locally. 
In terms of this new parameter $c>0$, the region $O_p$ is 
\begin{equation} \label{defn_Omegac}
\Omega_c:=\{ z \in X: \tilde{x}(z) \geq -c, \rho(z) \geq 0 \}.
\end{equation} 
We can choose $\tilde{x}$ such that $\bar{\Omega}_c$ is compact for $c$ sufficiently small.  Our proof for the local result is valid for all small $c$.

We give an explicit construction of $\tilde{x}$ here to show it exists locally
(thus can be used for our Theorem), 
but our result is valid for any $\tilde{x}$ satisfying conditions above.
Shrinking $O_p$ if necessary, we can assume the neighborhood we are working on is entirely in a local coordinate patch. We take
\begin{equation}
\tilde{x}(z) = -\rho(z) - \epsilon |z-p|^2, \quad  z \in O_p,  \label{tildex}
\end{equation}
where $|\cdot|$ means the Euclidean norm in this coordinate patch, and this term is introduced to enforce the region characterized by $\tilde{x}$ to be compact. Here $\epsilon>0$ is a fixed constant chosen before we choose $c$. For example, we can take $\epsilon = 1$.
Taking $c>0$ sufficiently small, $\{\tilde{x}>-c\}$ is compact. This is because $\tilde{x}>-c,\rho \geq 0$ implies $\rho \leq C$ and $|z-p|\leq c/\epsilon$. Since we are in a fixed coordinate patch, topologically this region is a closed subset of a compact Euclidean ball, hence it is compact. 
In addition, by the discussion in \cite[Section~3.1]{uhlmann2016inverse},
each $\Sigma_t$ with $0 \leq t \leq c$ is convex in the sense that any geodesic tangent to it curves away from $\{ \tilde{x} \leq -t \}$.

If we define $\Omega_c$ to be $\{0 \leq \rho(z) \leq c\}$, the region might be non-compact (even when $c$ is small, it might be a long thin strip near the boundary). So we use a modification of $-\rho$ making the level sets less convex to enforce its intersection with $\partial X$ happen in a compact region. 
Furthermore, the class of `adapted' function defined below is determined by the foliation given by $\tilde{x}$, which makes $\tilde{x}$ even more important in two dimensional case compared with higher dimensional cases.

We now turn to the \textit{convex foliation condition} we need in the two dimensional case. 
From now on, we assume $\tilde{x}$ to be any function that satisfies (\ref{defn_tildex}) 
and convexity condition after it.
Our foliation of the part of $X$ near $\partial X$ is given by level sets of $\tilde{x}$. 
That is, the family of hypersurfaces $\{\tilde{\Sigma}_t = \tilde{x}^{-1}(-t),0 \leq t \leq T\}$.
Here we choose $T$ to be a number such that desired properties of $\tilde{x}$ hold from $\tilde{\Sigma}_0$
up to $\tilde{\Sigma}_T$. 
For the Theorem, which concerns the local injectivity,
we only use the part of the foliation with $t \leq c$. While for the Corollary,
which concerns the global result, we use the entire foliation up to $\tilde{x}=-T$.
By our choice of $c$, we may take $T=c$. Taking $T$ larger will make the region 
on which our result holds larger. In fact, one may apply a layer stripping method 
to obtain injectivity result up to $\tilde{\Sigma}_T$.
When $T>c$, we may take $\tilde{\Sigma}_c$ as the `new boundary' and apply our Theorem,
 and then repeat. For more details of the layer stripping method, see
  the discussion after \cite[Corollary]{uhlmann2016inverse}.

Next we define the adapted function class associated to $\tilde{x}$.
\begin{defn}
With notations above, $\mathcal{F}_{\tilde{x}}(X)$ is defined to be the function space consists of functions which are constant on each $\tilde{\Sigma}_t$, and we say such a function is \textit{`adapted to the foliation $\tilde{x}$'}. In addition, $\mathcal{F}_{\tilde{x}} ^s(O_p):=\mathcal{F}_{\tilde{x}}(X) \cap H^s(O_p)$, where $H^s(O_p)$ denotes Sobolov space of order $s$, defined by identifying $O_p$ with $\bar{\R}^n$. 
\end{defn}
This function class depends on the choice of $\tilde{x}$. Our result is valid for any fixed $\tilde{x}$, and corresponding class of functions adapted to the foliation given by it.

Finally, the coordinate system we use is 
\begin{equation} \label{coordinate}
(x,y,\xi,\eta), 
\end{equation}
where $x=\tilde{x}+c$, and $y$ is the coordinate on $\tilde{\Sigma}_t$, which are 
line segments in our context. $\xi,\eta$ are fiber variables dual to $x,y$ respectively
in the scattering cotangent bundle $^{\;sc}T^*X$
(i.e., we are writing covectors as $\xi \frac{dx}{x^2}+\eta \frac{dy}{x}$).

We emphasize that introducing $\tilde{X}$ and the artificial boundary $\{\tilde{x}=-c\}$ brings us convenience in this framework, allowing us to 
restrict our analysis in to this local region and use the scattering pseudodifferential calculus.


\subsection{The geodesic ray transform} \label{sec_defn_transform}
Geodesics below are with respect to the metric $g$. 
Recalling (\ref{defn_Omegac}), we replace $\tilde{x}$ by $x=\tilde{x}+c$, so that $x$ itself becomes the defining function of the artificial boundary. In an open set $O \subset \bar{X}$, for a geodesic segment $\gamma \subset O$, we call it \textit{O-local geodesic} if its endpoints are on $\partial X$, and \textbf{all geodesic segments we consider below are assumed to be $O_p-$local.}
Next we introduce strictly positive density functions, which is needed in the discussion of geodesic ray transform. We use ${TX}$ to denote the tangent bundle of $X$, and notice that each point on it (i.e., a point with a tangent vector living at that point) determines a geodesic. Before defining the function class, we define $G_X := \{ (s,{z}) | s\in {TX}, {z} \in X \text{ lies on the geodesic determined by } s \}$. $G_X$ is a submanifold of ${TX} \times X$.

\begin{defn}
$\varrho \in \mathcal{C}^{\infty}(G_X)$ is called a $O_p-$strictly positive density function if:
\begin{enumerate}
\item For $s_1,s_2 \in {TX}$, if they determine the same $O_p-$local geodesic $\gamma$, then $\varrho(s_1,{z})=\varrho(s_2,{z})$ for ${z} \in \gamma$.
\item $C_1 \geq \varrho \geq C_0 >0$ on $({TX}|_{O_p} \times O_p )\cap G_X$ for some constants $C_0,C_1$.
\end{enumerate} \label{defn_weight1}
\end{defn}

The upper bound condition, which is assumed to ensure integrability, can be weakened.   
As we mentioned in the introduction, \textit{the local geodesic ray transform weighted by $\varrho$} of a function $f$  is defined by:
$$ I_\varrho  f(s) := \int_\gamma \varrho(s,\gamma(t)) f(\gamma(t))dt,$$
where $\varrho$ is an $O_p-$strictly positive density function, $s \in {PSX}$, $\gamma(\cdot)$ is the geodesic determined by $s$.
So our geodesic ray transform is a function on ${PSX}$. Condition 1. above implies that for $s_1,s_2$ projecting to the same pont on $X$ and has parallel nonzero fibre part, they determine the same function $\varrho(s_1,\cdot)=\varrho(s_2,\cdot)$ of $z$.

Our local injectivity result for the geodesic ray transform implies 
the local injectivity for the unweighted geodesic ray transform over the same class of functions
as here, i.e., those ones that are adapted to a foliation. 
To reduce to the unweighted case, we take $\varrho \equiv 1$. 
The purpose of adding this notion of density function is to make our theorem more general, and $\varrho$ here should be considered as `known' and our injectivity claim is for $f$ only. 
For density functions without condition 1 of Definition \ref{defn_weight1}, i.e., for density functions may also depend on the choice of the starting point of the geodesic, we have more information since our given information in the injectivity problem is the vanishing of these geodesic ray transforms. In that case, we have many integrals for a single geodesic. This means our formulation is the case where we need the `least information'.  
On the other hand, the injectivity of the geodesic ray transform for general class
of functions remains open, see \cite{paternain2022geometric} for more 
recent progress.

\subsection{The main result}
We use exponentially weighted Sobolev spaces: $H^s_F(O_p) : = e^{\Phi_F(x)}H^s(O_p) =\{ f \in H_{loc}^s(O_p): e^{-\Phi_F(x)}f \in H^s(O_p)\}$, where the additional subscript $F$ is a positive constant, which indicates the exponential conjugation, $\Phi_F(x)=\frac{F}{x}$ when $x$ is close to 0, and $\Phi_F(x)=\frac{F}{c-x}$ when $x$ is close to c.

For exponentially weighted Sobolev spaces on other manifolds, we use the same notation with $O_p$ replaced by that manifold. 
Furthermore, $PSX|_{O_p}$ is the restriction of the projective sphere bundle to $O_p$.
 With all these preparations, the main theorem is:
\begin{thm*}
For $p \in \partial X$, with density $,\tilde{x}, \varrho$ as above, we can choose $O_p = \{\tilde{x} > -c \}\cap \bar{X}$, so that the local geodesic transform is injective on $\mathcal{F}_{\tilde{x}} ^s(O_p) := \mathcal{F}_{\tilde{x}}(X) \cap H^s(O_p)$, $s\geq 0$. More precisely, there exists $C>0$ such that for all $f \in \mathcal{F}_{\tilde{x}} ^s(O_p)$, 
\begin{align} \label{est_main}
||Rf||_{ H^{s}_F ([0,c]_{x})} \leq C || I_\varrho f ||_{H^{s+\frac{3}{2}}(PSX|_{O_p})},
\end{align}
where $[0,c]_x$ is the part of a line transversal to $\tilde{x}-$level sets 
with $x: = \tilde{x}+c \in [0,c]$.
\end{thm*}
In the corollary below, $X,\tilde{\Sigma}_t$ are defined as above, and in addition we assume that $\bar{X}$ is compact.
\begin{coro*}
If the convex foliation construction $\{\tilde{\Sigma}_t\}$ is valid up to $\{\tilde{x}=-T\}$ and
$K_T : = X \backslash \cup_{t \in [0,T)} \tilde{\Sigma}_t$ has measure zero, the global geodesic X-ray transform is injective on $\mathcal{F}^0_{\tilde{x}}(X):=\mathcal{F}_{\tilde{x}}(X)\cap L^2(X)$. If $K_T$ has empty interior, the global geodesic transform is injective on $\mathcal{F}_{\tilde{x}} ^s(X):=\mathcal{F}_{\tilde{x}}(X)\cap H^s(X)$ for $s > \frac{n}{2}$.
\end{coro*}

\begin{rmk*}
We added `global' because the function are not restricted to $O_p$ anymore. The geodesic X-ray transform appears in this paper is weighted. 
In addition, our result is stable since all the conditions we need in the proof
are also satisfied by small perturbations of $g$, and the constant $C$ 
in (\ref{est_main}) can be made uniform for small perturbations of $g$,
hence the same result holds for small metric perturbations.
\end{rmk*}

\begin{proof} Assuming the theorem holds, we prove the corollary. 
For nonzero $f \in L^2(X)$ and $K_T$ has measure zero case, $\text{supp}f$ has non-zero measure by the definition of $L^2(X)$.
Consider $\tau:=\text{inf}_{\text{supp}f} (-\tilde{x})$. If $\tau \geq T$, then $\text{supp} f \subset K_T$, which has measure zero, contradiction. So $\tau < T$ and by definition $f \equiv 0$ on $\tilde{\Sigma}_t$ with $t < \tau$. By the definition of $\tau$, closedness of $\text{supp}f$ and compactness of $\bar{X}$, we know there exists $q \in \tilde{\Sigma}_\tau \cap \text{supp}f$. However, consider the manifold given by $\{ \tilde{x} < -\tau \}$, to which we can apply our theorem. Since we have local injectivity near $q$, we conclude that $q$ has a neighborhood disjoint with $\text{supp}f$, contradiction.

If $f \in H^s(X),s>\frac{n}{2}$, $f \neq 0$, then $f$ is continuous by the Sobolev embedding theorem and consequently supp$f$ has non-empty interior since .Then apply local result to a fixed point in $\text{supp}f$ gives the contradiction.
\end{proof}

\section{Sobolev spaces and the scattering calculus}
\label{calculus}
In this section, we recall some basic facts of pesudodifferential operators, their symbols, and the process of quantization, and also some basic facts about Sobolev spaces.
\subsection{Sobolev spaces}
 We state some inclusion relationship between weighted Sobolev spaces. Suppose $\bar{M}$ is a compact manifold with boundary whose interior is $M$. Let $\mathcal{V}_b(\bar{M})$ be the collection of all smooth vector fields tangent to $\partial M$. Suppose $x$ is a global boundary defining function, we set $\mathcal{V}_{sc}(\bar{M}) = x\mathcal{V}_b(\bar{M})$.

Then the $L^2$-integrability with respect to the scattering density $x^{-(n+1)}dxdy$ gives $L^2_{sc}(\bar{M})$. Here the density comes from the identification through $x = {r_1}^{-1}$, and the ordinary volume form in the polar coordinate is ${r_1}^{n-1}d{r_1}dy$, where $y$ denotes the spherecal variables and ${r_1}$ denotes the radial variable. The corresponding polynomially weighted Sobolev space $H_{sc}^{s,r}(\bar{M})$ consists of functions $u$ such that $x^{-r}V_1V_2...V_ku \in L^2_{sc}(\bar{M})$ for $k \leq s$ (when $k=0$, it's $u$ itself), and $V_j \in \mathcal{V}_{sc}(\bar{M})$.

With these definitions, we know
\begin{align*}
&H^s(\bar{M}) \subset H^{s,r}_{sc}(\bar{M}), \, r \leq - \frac{n+1}{2} \\
&H^{s,r'}_{sc}(\bar{M}) \subset H^s(\bar{M}), \, r' \geq - \frac{n+1}{2} + 2s.
\end{align*}
See Section 2.3 of \cite{uhlmann2016inverse} for more details.

\subsection{Scattering calculus on the Euclidean space}
$a(z,\zeta) \in \mathcal{C}^{\infty}(\R^n_z \times \R^n_\zeta)$ is said to be a scattering symbol of order $(m,l)$ if and only if:
$$ |D_z^\alpha D_\zeta^\beta a(z,\zeta)| \leq C_{\alpha \beta} \langle z \rangle^{l-|\alpha|} \langle \zeta \rangle^{m-|\beta|},$$
where $\langle z \rangle = (1+|z|^2)^{\frac{1}{2}}$, with $|z|$ being the Euclidean norm, and similarly for $\langle \zeta \rangle $. The space consists of such symbols is denoted by $S^{m,l}(\R^n,\R^n)$, or $S^{m,l}$ for short. Then the space of pesudodifferential operators $\Psi^{m,l}_{sc}(\R^n)$ is defined as `the left quantization' of such symbols.

In order to facilitate the discussion of the quantization map and the generalization to general manifolds with boundary, we compactify the $\R^n$ in both base and phase factors. Concretely, we compactify $\R^n$ to a closed ball $\bar{\R}^n$ by adding the `sphere at infinity' $\mathbb{S}^{n-1}$. Using ${r_1}$ to denote the radial variable, we first identify $\R^n\backslash\{0\}$ with $(0,+\infty)_{r_1} \times \mathbb{S}_\theta^{n-1}$ through polar coordinates $({r_1},\theta) \rightarrow {r_1}\theta$. Let $x={r_1}^{-1}$, then $\R^n\backslash\{0\}$ becomes $(0,+\infty)_x \times \mathbb{S}_\theta^{n-1}$. Now glueing a sphere to $x=0$, or extending the range of $x$ to $[0,\infty)$ is equivalent to attaching a sphere at infinity in the original coordinates. So formally $\R^n$ is obtained by taking disjoint union of $\R^n$ and $[0,+\infty)_x \times \mathbb{S}^{n-1}$ modulo the identification given above. Now $x={r_1}^{-1}$ is a boundary defining function near $\partial \bar{\R}^n$. By modifying it in the `large $x$ small ${r_1}$' part, this gives us a global boundary defining function $\rho$. Decay properties can be rephrased as regularity on this compatified space: Schwartz functions on $\R^n$ are exactly restrictions to $\R^n$ of $\mathcal{C}^{\infty}$ functions on $\bar{\R}^n$.

Now we return to give an explicit formula of the left quantization. Let $y$ be a coordinate system on $\mathbb{S}^{n-1}$
in the decomposition $[0,+\infty)_x \times \mathbb{S}^{n-1}$. 
For $a \in S^{m,l}(\R^n,\R^n)$, in terms of the coordinates after the radial compactification above, 
the left quantization of $a$ is the operator defined by
\begin{equation} \label{defn_quantization_1}
\tilde{q}_L(a)u(x,y) = (2\pi)^{-n} \int e^{i(\xi\frac{x-x'}{xx'}+\eta (\frac{y}{x}-\frac{y'}{x'}))}u(x',y')a(x,y,\xi,\eta)\frac{dx'dy'}{(x')^{n+1}} d\xi d\eta.
\end{equation}
This is the left quantization of scattering pseudodifferential operators
given by Melrose \cite{melrose1994spectral}, but written in a different manner,
while giving the same operator algebra.
Recall that the scattering double space is obtained from blowing up
the double space $\bar{\R}^n \times \bar{\R}^n$ twice, first along the 
corner $(\partial \bar{\R}^n)^2$, and then along the boundary of
the lifted diagonal after this first blow up. 
The reason we choose this quantization formula is that 
the definition (\ref{defn_quantization_1}) we are using
is more convenient for our purpose since it is valid not only near the lifted diagonal
of the scattering double space, but also away from it, `globally' in a coordinate chart.
From a geometric point of view, this corresponds to choosing a different 
parametrization near the lifted diagonal. Concretely, we use
\begin{align} \label{coordinate_sc_diagonal}
x,y,\tilde{X} = \frac{x'-x}{xx'}, \tilde{Y} = \frac{y}{x}-\frac{y'}{x'}
\end{align}
as a coordiante system, which is valid up to the region where the coordinate systems $(x,y),(x',y')$ are valid. On the other hand, the more commonly seen construction of the scattering calculus is using
\begin{align} \label{coordiante_sc_diagonal_2}
x,y,X = \frac{x-x'}{x^2},Y= \frac{y-y'}{x},
\end{align}
to parametrize the region near the lifted diagonal in the scattering double space. The corresponding quantization map is given by
\begin{equation} \label{defn_quantization_2}
q_L(a)u(x,y) = (2\pi)^{-n} \int e^{i(\xi\frac{x-x'}{x^2}+\eta \frac{y-y'}{x})}u(x',y')a(x,y,\xi,\eta)\frac{dx'dy'}{(x')^{n+1}} d\xi d\eta.
\end{equation}
See \cite[Section~2.1]{zachos2022inverting} for a detailed discussion on this.

Both the space of symbols and that of pseudodifferential operators increase with respect to $m,l$. This family of spaces $\Psi^{*,*}_{sc}(\R^n)$ forms a filtered $*-$algebra under the composition and taking adjoints relative to the Euclidean metric, i.e.,
\begin{align*}
A \in \Psi^{m,l}_{sc}(\R^n), B \in \Psi^{m,l}_{sc}(\R^n) \implies AB \in \Psi^{m+m',l+l'}_{sc}(\R^n),
\end{align*}
and
\begin{align*}
A \in \Psi^{m,l}_{sc}(\R^n) \implies A^* \in \Psi^{m,l}_{sc}(\R^n).
\end{align*}

The next important notion is the principal symbol. For $A \in \Psi^{m,l}_{sc}$, its principal symbol is the equivalence class of $a$ in $S^{m,l}/S^{m-1,l-1}$ where $a$ is the symbol whose left quantization is $A$. This equivalence class captures the behaviour and properties of $A$ modulo lower order operators. We say $A \in \Psi^{m,l}_{sc}(\R^n)$ is elliptic if its principal symbol is invertible in the sense that there exists $b \in S^{-m,-l}$ such that $ab-1 \in S^{-1,-1}$. Whether $b$ exists or not does not depend on the choice of representative of $a$ in that class. When $A$ is elliptic, the standard parametrix construction gives us $B \in \Psi^{-m,-l}_{sc}(\R^n)$ such that $AB-\Id \in \Psi^{-\infty,-\infty}_{sc}(\R^n)$, which means the error term operator has a Schwartz function on $\R^{2n}$ as its Schwartz kernel. Hence $R:=AB-\Id$ mapping $H^{s,r}$ to $H^{s-m,r-l}$ is compact for any $s,r,m,l$, with $H^{s,r}$ defined in the Sobolev space section. Compactness and parametrix construction together gives Fredholm property of elliptic operator $A$. That is, it has closed range, finite dimensional kernel and cokernel, with following estimate for any $N$:
$$ ||u||_{H^{s,r}(\R^n)} \leq C (||Au||_{H^{s-m,r-l}(\R^n)} + ||\tilde{F}u||_{H^{-N,-N}(\R^n)}),$$
where $\tilde{F}$ can be taken as a finite rank operator in $\Psi^{-\infty,-\infty}_{sc}(\R^n)$.

As we have mentioned, we can campactify both factors of $\R_z^n \times \R_\zeta^n$ to define the scattering symbols of $\bar{\R}_z^n \times \bar{\R}_\zeta^n$. We denote the defining function of $\bar{\R}_z^n$ (`the position factor') by $\rho_\partial$, and that of $\bar{\R}_\zeta^n$ (`the momentum factor') by $\rho_\infty$. We also define $\Psi^{m,l}_{sc}(\bar{\R}^n):=\Psi^{m,l}_{sc}(\R^n)$. The ellipticity of $A \in \Psi^{m,l}_{sc}(\R^n)$ is equivalent to the non-vanishing property of $\rho_\partial^l\rho_\infty^ma$, where $a$ is a left symbol (whose left quantization is that operator) of $A$. In particular, by the compactness of the boundary sphere, there exists $C>0$ such that $|a| \geq  C \rho_\partial^{-l}\rho_\infty^{-m}$, which is convenient to use in practice.

\subsection{Generalization to manifolds}
Let $\bar{M}$ be a manifold with boudary and denote its interior by $M$. Then the scattering pseudodifferential operators $\Psi^{m,l}_{sc}(\bar{M})$ is obtained by locally identifying the manifold with $\bar{\R}^n$. On such charts $U \times U$, the Schwartz kernel of the operator has the same property as the case of $\bar{\R}^n$, and we also allow additional globally Schwartz terms in the Schwartz kernel. All those algebraic properties of $\Psi^{m,l}_{sc}(\bar{\R}^n)$ generalize to the manifold case. In addition, the weighted Sobolev spaces $H^{s,r}_{sc}(\bar{M})$ are also obtained by locally identifying $\bar{M}$ with $\bar{\R}^n$.
A clarification on how we define $L^2_{sc}(\bar{M})$ might be useful. After locally identify $\bar{M}$ with $\bar{\R}^n$, we use the scattering density ${r_1}^{n-1}d{r_1}dy=x^{-n-1}dxdy$ to define $L^2_{sc}$ on this coordinate patch, then we use a partition of unity to define $L^2_{sc}(\bar{M})$ space. One can check that membership of this spaces is independent of the choice of the partition of unity.

\section{The combined pseudodifferential algebra and Sobolev spaces}
\label{sec: combined psiDO}
In this section we introduce the combined pesudodifferential algebra
that we will use, and prove a lemma the we need in showing the symbolic property 
of our modified normal operator.
Our combined pseudodifferential algebra consists of operators that 
are locally scattering pseudodifferential operators near
$0 \in [0,c]$ and $c \in [0,c]$. Concretely, we define:
\begin{defn}
The operator class $\Psi_{sc,sc}^{m,l_1,l_2}([0,c])$ consists of operators $A$ such that 
\begin{itemize}
	\item For $\phi,\psi \in  C_c^{\infty}(\R)$ with support away from $0$, $\phi A \psi \in \Psi_{sc}^{m,l_1}((0,c])$, and the scattering algebra is constructed using $c$ as the boundary.
	\item For $\phi,\psi \in  C_c^{\infty}(\R)$ with support away from $c$, $\phi A \psi \in \Psi_{\mathrm{sc}}^{m,l_2}([0,c))$, and the scattering algebra is constructed using $0$ as the boundary.
	\item For $\phi,\psi \in  C_c^{\infty}(\R)$ with disjoint support, $\phi A \psi$ has Schwartz kernel which is $C^\infty$ and rapidly vanishing when approaching both
	$0$ and $c$.
\end{itemize}
\label{defn_operator}
\end{defn}
One can check that
$\Psi_{{sc},{sc}}^{\infty,\infty,\infty}([0,c]) = \cup_{m,l_1,l_2 \in \Z} \Psi_{{sc},{sc}}^{m,l_1,l_2}([0,c])$ is a
tri-filtered $*-$algebra. We also need Sobolev spaces; as usual these
are defined by localization:
\begin{defn}
The function class $H_{{sc},{sc}}^{m,l_1,l_2}([0,c])$ consists of functions $f$ such that 
\begin{itemize}
	\item For $\phi \in  C_c^{\infty}(\R)$ with support
          away from $0$, $\phi f  \in H_{{sc}}^{m,l_1}((0,c])$, and the scattering Sobolev space is constructed using $c$ as the boundary.
	\item For $\phi \in  C_c^{\infty}(\R)$ with support
          away from $c$, $\phi f  \in H_{{sc}}^{m,l_2}([0,c))$, and the scattering Sobolev space is constructed using $0$ as the boundary.
\end{itemize}
\label{defn_function}
\end{defn}

We also define a new cotengent bundle that has corresponding boundary behavior near each boundary. For convenience, we denote the defining function on of the boundary
$x=c$ as 
\begin{align*}
\rho_c:=c-x.
\end{align*}
Since the difference of two coordinate patches of $X$ near $x=0$
and $x=c$ is just using $x$ or $c-x$ as the first variable,
we use the same notations for points on $TX$
on two patches near $x=0$ and $x=c$: we still write vectors as $-\lambda \partial_{\rho_c}+\omega \partial_y$; we also use $(x,y,\lambda,\omega)$
to indicate starting points of $\gamma_{x,y,\lambda,\omega}(t)$, and also the first variable of the weight $\varrho$.
On the other hand, since we use scattering cotangent bundle near each end using $x,\rho_c$ as defining functions respectively, we use $\xi_{sc,0},\xi_{sc,c}$ as fiber variables of scattering cotagent bundles near $x=0,x=c$ respectively.

\begin{defn}
The sc-sc cotangent bundle $^{{sc},{sc}}T^*[0,c]$ is the vector bundle such that
\begin{itemize}
\item In a neighborhood of $c$, its local frame is given by the scattering cotangent bundle with $c$ being the boundary: $\frac{dx}{\rho_c^2}$.
\item In a neighborhood of $0$, its local frame is given by the frame of the scattering cotangent bundle with $0$ being the boundary:  $\frac{dx}{x^2}$.
\item Away from both $0,c$, its local frame is the
  same as the usual cotangent bundle: $dx$.
\end{itemize}
\label{defn_bundle}
\end{defn}

\section{Ellipticity of the normal operator} 
\label{osc}
In this section we first reformulate the geodesic X-ray transform on our function class using pesudodifferential operators on the transversal line we choose:
$$\{x \in [0,c], y=0\}.$$ 
Then we prove the ellipticity of the exponentially conjugated microlocalized normal operator. The exponential conjugation is needed because although the Schwartz kernel of $A=L \circ I_\varrho$ behaves well when $X=\frac{x'-x}{x^2},\, Y=\frac{y'-y}{x}$ are bounded, it is not so when $(X,Y) \rightarrow \infty$. This conjugation gives additional exponential decay to resolve this issue, and also regularizes
the Schwartz kernel of $A_{1,F}$, which is the `one dimensional version'
of $A$.

Using the notation $(z,\nu) = (x,y,\lambda,\omega) \in PSX$ and $\rho_c=c-x$,  
we define the modified adjoint operator of $I_\varrho$ as
$$
(Lv)(z) := x^{-2} \int \chi(x,\frac{\lambda}{x},\frac{\lambda}{\rho_c}) v( \gamma_{x,y,\lambda,\omega}) d\lambda d\omega,
$$
where $\gamma_{x,y,\lambda,\omega}(t)$ is the geodesic starting at $(x,y)$ with initial tangent vector $(\lambda,\omega),\omega=\pm1$.
We write it in component form as $\gamma_{x,y,\lambda,\omega}(t) = (\gamma^{(1)}_{x,y,\lambda,\omega}(t),\gamma^{(2)}_{x,y,\lambda,\omega}(t))$.
We will use $\gamma(t)$ to indicate $\gamma_{x,y,\lambda,\omega}(t)$ when there is no confusion about the choice of its starting point in the context.
We choose $\chi$ to be:
\begin{align*}
\chi=\chi_0(x)\tilde{\chi}(\frac{\lambda}{x})+ (1-\chi_0(x))\tilde{\chi}(\frac{\lambda}{\rho_c}),
\end{align*}
where $\chi_0$ is 1 when $x<c/3$ and supported in $(-\infty,\frac{c}{2}]$;
$\tilde{\chi}$ is a compactly supported even function, strictly positive near 0. 
In particular, when $x<c/3$, $\chi(x,\frac{\lambda}{x},\frac{\lambda}{\rho_c})$ can be written as a function of $x,\frac{\lambda}{x}$ with compact support; and when $x>\frac{c}{2}$, it can be written as a function of $x,\frac{\lambda}{\rho_c}$ with compact support.
$v$ is a function defined on the space of $O_p-$local geodesic segments, whose prototype is the geodesic ray transform
$$
v(\gamma) = I_\varrho f(\gamma) = \int_\gamma f(\gamma(t)) \varrho((x,y,\lambda,\omega),\gamma_{x,y,\lambda,\omega}(t))dt,
$$
in which $f( \gamma(t) )$ can be replaced by higher order tensors, coupling with $\dot{\gamma}(t)$ in all of its slots in more general situations. 
By the compactness of $\bar{\Omega}_c$ discussed after (\ref{tildex}) and $|(\lambda,\pm1)|\geq 1$ and the convexity assumption on 
we made, \cite[Lemma~3.1]{vasy2020semiclassical} shows that there exits a uniform bound $T_g$ of the escape time of $\bar{O}_p$. Thus 
we assume $|t|\leq T_g$ in arguments below.
$I_{\varrho}$ is the original geodesic ray transform operator and $L$ is its adjoint if we ignore $\chi$ and assume fast decay conditions on integrands. So their composition is the model of the normal operator.  
We define the conjugated normal operator $A_F$ as
\begin{equation*}
A_F = e^{-\Phi_F(x)}L \circ I_\varrho e^{\Phi_F(x)},
\end{equation*}
where $\Phi_F(x)=\frac{F}{x}$ when $x$ is close to 0 and $\Phi_F(x)=\frac{F}{\rho_c}$ when $x$ is close to c. We only concern the case where $\Phi_F$ is given by one of these two expressions, since in the region where $x$ is away from both $0$ and $c$, this conjugation only introduces a smooth and lower bounded weight, and we can write (after localized to this region):
\begin{align*}
e^{i\Phi_F(x)}\varrho = e^{iF/x} \varrho',
\end{align*}
then $\varrho'$ is a smooth positive multiple of $\varrho$ and satisfies all requirements for $\varrho$. Then we can apply the argument for $\Phi_F=F/x$ to this region.

By the definition of $L$ and $I_\varrho$, we know $A_F$ acts by
\begin{align*}
A_Ff(z):= & x^{-2}\int e^{-\Phi_F(x)+ \Phi_F(x(\gamma_{x,y,\lambda,\omega}(t)))} \chi(x,\frac{\lambda}{x},\frac{\lambda}{\rho_c}) f(\gamma_{x,y,\lambda,\omega}(t))
\\ & \varrho((x,y,\lambda,\omega),\gamma_{x,y,\lambda,\omega}(t)) dt |d\nu|,
\end{align*}
with $|d\nu|=|d\lambda d\omega|$ being a smooth density. 

Next we introduce the `one-dimensional version' of $A_F$:
\begin{align}
A_{1,F} = R \circ A_F \circ \chi_XE.
\end{align}
$E$ extends the function on $\{y=0\}$ to a function on $O_p$ which is constant on each level sets of $\tilde{x}$, $R$ is the restriction operator restricting functions on $O_p$ to $\{y=0\}$; and $\chi_X$ means multiplying by the characteristic function of $X$ after applying $E$.

Recall the parameter $c$ in the definition $O_p=\{\tilde{x} >-c \} \cap \bar{X}$ and we only concern the region $0 \leq x = \tilde{x}+c \leq c$,

\subsection{Ellipticity near the \texorpdfstring{$x=0$}{x=0} }
In this part we show the ellipticity 
of $A_{1,F}$ near the end point $x=0$,
including the ellipcity at the fiber infinity and that of the boundary symbol.

\begin{prop} \label{AFproperty}
$A_F \in \Psi_{sc}^{-1,0}([0,c)_x)$ for $F>0$, where the scattering pseudodifferential algebra is constructed using $x=0$ as the boundary. In addition, if we choose $\chi \in \mathcal{C}^\infty$ appropriately with $\chi \geq 0, \chi(x,0,\frac{\lambda}{\rho_c})=1$, and compactly supported with respect to the second variable when $x<\frac{c}{4}$, its principal symbol, including the boundary symbol, is elliptic.
\end{prop}
\begin{proof}
The proof for the pseudodifferential property of $K_{A_F}(x,y,x',y')$ in dimension two is the same as that in dimension $\geq 3$. 
Thus for the derivation of the decay property of $A_F$'s Schwartz kernel and consequently the membership $A_F\in \Psi_{sc}^{-1,0}$ (this scattering pseudodifferential algebra is for the two dimensional region $\{ x \geq 0\}$, taking
$\{x=0\}$ as the boundary surface), we refer readers to \cite[Section~3.5]{uhlmann2016inverse} or \cite[Section~5]{stefanov2004stability}. 
For our $A_{1,F}$, the effect of the extension (by constant) operator to the Schwartz kernel is integrating $K_{A_F}$ against $y'$. 
The restriction operator is evaluating $K_{A_F}$ to $y=0$; both of these two operations do not affect the regularity (conormality in the blown up scattering double space, in fact we concern only the part for $x$) of the Schwartz kernel.
Multiplying $\chi_X(x',y')$ effectively introduces the exist times $\tau_\pm(x,0,\lambda,\omega)$ as upper and lower bounds of the $t-$integral as in (\ref{schwartzkernel1}) below. Since $\tau_\pm$ are smooth away from $x=c$ (i.e., the point $p$), thus the pseudodifferential property is not affected on $[0,c)_x$.

Next we focus on ellipticity here. Since we only concern principal symbol level information, we use (\ref{defn_quantization_2}) as the definition of the quantization, which is more convenient to evaluate in the current setting and does not differ from (\ref{defn_quantization_1}) on the principal symbol level, to compute the principal symbol.

From the definition of $A_{1,F}$, we know its Schwartz kernel is
\begin{align*}
K_{A_{1,F}}(x,x') = \int & e^{-\Phi_F(x) + \Phi_F(x(\gamma_{x,0,\lambda,\omega}(t))) } x^{-2}\chi(x,\frac{\lambda}{x},\frac{\lambda}{\rho_c})\delta(z'-\gamma_{x,0,\lambda,\omega}(t))  
\\ & \varrho((x,0,\lambda,\omega),\gamma_{x,y,\lambda,\omega}(t)) \chi_X(\gamma(t)) dt |d\nu| dy'
\\ = (2 \pi)^{-n} \int & e^{-\Phi_F(x) + \Phi_F(x(\gamma_{x,0,\lambda,\omega}(t))) }    x^{-2}\chi(x,\frac{\lambda}{x},\frac{\lambda}{\rho_c})  e^{-i (\xi',\eta') \cdot ((x',y')-\gamma_{x,0,\lambda,\omega}(t))} 
 \\ & \varrho((x,0,\lambda,\omega),\gamma_{x,0,\lambda,\omega}(t)) \chi_X(\gamma(t))  dt |d\nu|dy'd\xi'd\eta'.
\end{align*}
Evaluating the $\eta'-$integral, giving $\delta(\gamma^{(2)}_{x,0,\lambda,\omega}(t)-y')$, and the $y'-$integral can be removed (the integrand other than this delta function is independent of $y'$), and the oscillatory factor is $e^{-i\xi'(x-\gamma_{x,0,\lambda,\omega}(t))}$.
By introducing the exit time, we can remove the $\chi_X(\gamma(t))$ factor by restricting the range of integration of $t$.
\begin{defn}
We define $\tau_+(x,y,\lambda,\omega) \geq 0$ to be the travel time (in our coordinate patch) from $(x,y)$ to $\partial X$ in the forward direction and $\tau_-(x,y,\lambda,\omega) \leq 0$
to be the travel time to $\partial X$ in the brackward direction.
\end{defn}

Then the Schwartz kernel can be written as
\begin{align}  \label{schwartzkernel1}
\begin{split}
K_{A_{1,F}}(x,x') = & (2 \pi)^{-n} \int_{\nu,\xi'} \int_{\tau_-(x,0,\lambda,\omega)}^{\tau_+(x,0,\lambda,\omega)} 
e^{-\frac{F}{x} +\frac{F}{x(\gamma_{x,0,\lambda,\omega}(t))}}   x^{-2}\chi(x,\frac{\lambda}{x},\frac{\lambda}{\rho_c})
 \\& e^{-i \xi' \cdot (x'-\gamma^{(1)}_{x,0,\lambda,\omega}(t))} 
 \varrho((x,0,\lambda,\omega),\gamma_{x,0,\lambda,\omega}(t))  dt |d\nu| d\xi'.
 \end{split}
\end{align}
Then convert this to the symbol by inverse Fourier transform in $x'$ evaluated at $\xi$, and this combined with the $e^{-i\xi' \cdot x' }$ gives $\delta_0(\xi-\xi')$, 
effectively replacing $\xi'$ in the expression by $\xi$,
we have
\begin{align*}
a_{1,F} =  &  (2 \pi)^n e^{-ix\xi} \mathcal{F}^{-1}_{x'\rightarrow \xi}K_{A_{1,F}}(x,x')  \\
  =  & \int_{\lambda,\omega} \int_{\tau_-(x,0,\lambda,\omega)}^{\tau_+(x,0,\lambda,\omega)}    e^{-\frac{F}{x} +\frac{F}{x(\gamma_{x,y,\lambda,\omega}(t))}} x^{-2}\chi(x,\frac{\lambda}{x},\frac{\lambda}{\rho_c})  e^{-i x {\xi} } e^{i {\xi} x(\gamma(t))}
  \\ & \varrho((x,0,\lambda,\omega),\gamma_{x,y,\lambda,\omega}(t))  dt  d\lambda d\omega.
\end{align*}
We write the coordinates in the scattering cotangent bundle near $x=0$ as 
\begin{align*}
\xi_{sc,0} \frac{dx}{x^2}.
\end{align*} 
The expression above become, with $\varrho$ standing for $\varrho((x,0,\lambda,\omega),\gamma_{x,0,\lambda,\omega}(t))$,
\begin{align}\label{eq3}
\begin{split}
a_{1,F}(x,\xi_{sc,0}) =  & \int_{\lambda,\omega} \int_{\tau_-(x,0,\lambda,\omega)}^{\tau_+(x,0,\lambda,\omega)} e^{-\frac{F}{x} +\frac{F}{x(\gamma_{x,0,\lambda,\omega}(t))}}  x^{-2}\chi(x,\frac{\lambda}{x},\frac{\lambda}{\rho_c})
\\& e^{i \frac{\xi_{sc,0} }{x^2} \cdot (\gamma^{(1)}_{x,0,\lambda,\omega}(t) - x) }  \varrho dt d\lambda d\omega.
\end{split}
\end{align}
Next we investigate the phase function of this oscillatory integral and then apply the stationary phase lemma. We denote components of $\gamma_{x,y,\lambda,\omega}(t)$ and its derivatives by 
\begin{align}
\begin{split}
&\gamma_{x,y,\lambda,\omega}(0) = (x,y), \quad \dot{\gamma}_{x,y,\lambda,\omega}(0) = (\lambda,\omega), \\
& \ddot{\gamma}_{x,y,\lambda,\omega}(t) = 2 (\alpha(x,y,\lambda,\omega,t), \beta(x,y,\lambda,\omega,t)),
\end{split}
\label{gamma}
\end{align}
where $\alpha,\beta$ are defined by this equation and are smooth with respect to their variables. In addition, $\alpha$ is a quadratic form in $\omega$, and it is strictly positive definite in $\omega$ for small enough $x,\lambda,t$, which means being close to the starting point at the boundary, by our convexity condition.

Since $\gamma_{x,y,\lambda,\omega}$ starts at $(x,y)$ with initial velocity $(\lambda,\omega)$, there exists
smooth functions $\Gamma^{(1)},\Gamma^{(2)}$ such that
$$
\gamma_{x,y,\lambda,\omega}(t) = ( x + \lambda t + \alpha t^2 + \Gamma^{(1)}(x,y,\lambda,\omega,t)t^3, y+ \omega t +  \Gamma^{(2)}(x,y,\lambda,\omega,t)t^2),
$$
where we only expand the second component to the first order and have included the $\beta-$term in the definition of $\Gamma^{(2)}$. Then we make the change of variables
$$
\hat{t}_0 = \frac{t}{x}, \quad \hat{\lambda}_0 =\frac{\lambda}{x}.
$$
By the support condition of $\chi$, the integrand is none-zero when $\hat{\lambda}_0$ is in a compact interval. However, the bound on $\hat{t}_0$ is $|\hat{t}_0| \leq \frac{T_g}{x}$, which is not uniformly bounded, we amend this by treating it in two regions separately. Using these new variables, we rewrite our phase as
$$
\phi = \xi_{sc,0}(\hat{\lambda}_0\hat{t}_0 + \alpha \hat{t}_0^2 + x\hat{t}_0^3 \Gamma^{(1)}(x,0,x\hat{\lambda}_0,\omega,x\hat{t}_0)).
$$
The damping factor coming from exponential conjugation is
\begin{align*}
-\frac{F}{x} + \frac{F}{\gamma^{(1)}_{x,0,\lambda,\omega}(t)} = & -F(\lambda t + \alpha t^2 + t^3 \Gamma^{(1)}(x,0,x\hat{\lambda}_0,\omega,x\hat{t}_0)) \\ 
& \times (x(x+\lambda t + \alpha t^2 +t^3 \Gamma^{(1)}(x,0,x\hat{\lambda}_0,\omega,x\hat{t}_0)))^{-1}   \\
=& -F(\hat{\lambda}_0\hat{t}_0 +\alpha \hat{t}_0^2 +\hat{t}_0^3x \hat{\Gamma}^{(1)}(x,0,x\hat{\lambda}_0,\omega,x\hat{t}_0)),
\end{align*}
where $\hat{\Gamma}^{(i)}$ is introduced when we first express $\gamma^{(1)}_{x,0,\lambda,\omega}(t)$ by variables $t,\lambda$, and then invoke our change of variables, then collect the remaining terms, which is a smooth function of these normalized variables. So this amplitude is Schwartz in $\hat{t}_0$, hence we take a constant $\epsilon_t>0$ and
deal with regions $|\hat{t}_0| \geq \epsilon_t$ and $|\hat{t}_0|< \epsilon_t$ separately. In our later arguments, we will take $\epsilon_t$ small
to enforce $\hat{t}_0=0$ holds for critical points.

Before considering the critical points of the phase for small $x>0$, we first consider the critical points of the phase at $x=0$. This helps us to get rid of 
those ${\Gamma}^{(i)}-$terms and simplifies the process to solve the equation for the critical points.

When $x=0$, the phase becomes
$$
\xi_{sc,0} (\hat{\lambda}_0\hat{t}_0 + \alpha \hat{t}_0^2).
$$
When $|\hat{t}_0| \geq \epsilon_t$, the derivative with respect to $\hat{\lambda}_0$ vanishes only when $\xi_{sc,0}=0$. 
Since our analysis on the ellipticity is happening away from the zero section, so the
lack of critical points means the region $|\hat{t}_0| \geq \epsilon_t$ gives rapid decay contribution. The case $x>0$ can be dealt with the same method, but with more complicated computation. Notice that, $\alpha,\Gamma^{(i)}$ take $\lambda = x\hat{\lambda}_0, t = x \hat{t}_0$ as variables, and produces an extra $x$ factor when we take partial derivatives with respect to $\hat{\lambda}_0,\hat{t}_0$. Concretely, the derivative with respect to $\hat{\lambda}_0$ is:
\begin{align*}
\frac{\partial \phi}{\partial \hat{\lambda}_0} &= \xi_{sc,0} \hat{t}_0(1+x\hat{t}_0\partial_\lambda \alpha + x^2 \hat{t}_0^2 \partial_\lambda \Gamma^{(1)})\\
&= \xi_{sc,0} \hat{t}_0(1+t\partial_\lambda \alpha + t^2 \partial_\lambda \Gamma^{(1)}).
\end{align*}
Recall that $|t| \leq T_g$ and we can choose $T_g$ to be small by shrinking $O_p$.
Thus when both $|\xi_{sc,0}| \geq C |\eta|$ and $|\hat{t}_0| \geq \epsilon_t$ hold,
$\xi_{sc,0}\hat{t}_0$ is non-zero and is going to dominate other terms, 
so $\frac{\partial \phi}{\partial \hat{\lambda}_0}$ can not vanish and there is no critical point in this case.
Since the $\epsilon_t$ in arguments above is arbitrary, we know that the condition $\hat{t}_0 = 0$ holds for any critical point including the $x \neq 0$ case.

Next we consider the region $|\hat{t}_0|<\epsilon_t$, whose closure is compact, and consequently we can apply the stationary phase lemma. The same as before, we consider the condition that the derivative with respect to $\hat{\lambda}_0$ and $\hat{t}_0$ vanish. First consider the $x=0$, in which case the expression can be significantly simplified:
$$
\xi_{sc,0} \hat{t}_0 = 0, \quad \xi_{sc,0} \hat{\lambda}_0 = 0.
$$
We exclude $\xi_{sc,0}=0$ case since we only concern ellipticity away from the zero section. 
Then we have the condition for critical points:
$$
\hat{t}_0 = 0, \quad \hat{\lambda}_0 = 0.
$$
The second condition can be derived if we notice that (for general $x$):
$$
\frac{\partial \phi}{\partial \hat{t}_0} = \xi_{sc,0}\hat{\lambda}_0  + O(\hat{t}_0),
$$
where the $O(\hat{t}_0)$ term vanishes when $\hat{t}_0=0$, and can be computed explicitly:
$$
2\xi_{sc,0} \alpha\hat{t}_0 + 3\rho_c \xi_{sc,0} \Gamma^{(1)} \hat{t}_0^2 + \xi_{sc,0} x^2 \partial_t\Gamma^{(1)} \hat{t}_0^3.
$$
So those two conditions for stationary points extends to the $x \neq 0$ case.
In order to apply those conditions of critical points of the phase, we first rewrite (\ref{eq3}) as:
\begin{align}
\begin{split}
&  a_{1,F}(x,\xi_{sc,0}) =  \int e^{-F(\hat{\lambda}_0 \hat{t}_0 + \alpha \hat{t}_0^2 + \hat{t}_0^3x \hat{\Gamma}^{(1)}(x,0,x\hat{\lambda}_0,\omega,x\hat{t}_0) )} \chi(x,\hat{\lambda}_0,\frac{\lambda}{c-x}) 
\\& \varrho((x,0,x\hat{\lambda}_0,\omega),
\gamma_{x,0,x\hat{\lambda}_0,\omega}(x\hat{t}_0)) e^{i\xi_{sc,0}(\hat{\lambda}_0 \hat{t}_0 + \alpha \hat{t}_0^2 + \hat{t}_0^3x {\Gamma}^{(1)}(x,0,x\hat{\lambda}_0,\omega,x\hat{t}_0))} d\hat{t}_0 d\hat{\lambda}_0d\omega,
\end{split}
\label{aF2}
\end{align}
where integrating over $\omega$ is just summing two terms at $\pm 1$. 
By stationary phase lemma with a non-degenerate critical point, the leading contribution comes from the critical points of the phase $\{\hat{t}_0=0, \hat{\lambda}_0 = 0 \}$. The $(\hat{t}_0,\hat{\lambda}_0)$-Hessian of the phase at the critical points is:
$$
\begin{pmatrix} 2 \alpha \xi_{sc,0} & \xi_{sc,0} \\ \xi_{sc,0} & 0  \end{pmatrix},
$$
which has determinant $-\xi_{sc,0}^2$. So the asymptotic behaviour of the integral as $|\xi_{sc,0}| \rightarrow \infty$ is the same as (up to a non-zero constant factor, and use the symmetry of $\varrho$ with respect to the vector fiber part)
\begin{align}
\label{symbol_main_term}
|\xi_{sc,0}|^{-1}\chi(x,0,0) \varrho( (x,0,0,\pm 1), (x,0) ) .
\end{align}
By our choice of $\chi$, $\chi(x,0,0)>0$, and we get a $-1$ order elliptic estimate for both $x>0$ and $x=0$.

Next we turn to show boundary part of the principal symbol of $A_{1,F}$ is also elliptic (when the fiber variable is finite). 
Evaluating (\ref{aF2}) at $x=0$, since $\omega = \pm 1$, the boundary principal symbol of $A_{1,F}$ is
\begin{align*}
a_{1,F}(0,\xi_{sc,0}) = &  \int e^{-F(\hat{\lambda}_0 \hat{t}_0 + \alpha \hat{t}_0^2)} \chi(\hat{\lambda}_0) e^{i \xi_{sc,0}(\hat{\lambda}_0 \hat{t}_0 + \alpha \hat{t}_0^2)} \varrho((0,0,0,\omega),\gamma_{0,y,0,\omega}(0))   d\hat{t}_0  d\hat{\lambda}_0d\omega\\
& = \varrho((0,0,0,1),(0,0))\int e^{-F(\hat{\lambda}_0 \hat{t}_0 + \alpha \hat{t}_0^2)} \chi(0,\hat{\lambda}_0,0) e^{i\xi_{sc,0}(\hat{\lambda}_0 \hat{t}_0 + \alpha \hat{t}_0^2 ) } d\hat{t}_0 d\hat{\lambda}_0  \\
&+ \varrho((0,0,0,-1),(0,0)) \int e^{-F(\hat{\lambda}_0 \hat{t}_0 + \alpha \hat{t}_0^2)} \chi(0,\hat{\lambda}_0,0) e^{i \xi_{sc,0}(\hat{\lambda}_0 \hat{t}_0 + \alpha \hat{t}_0^2 ) } d\hat{t}_0 d\hat{\lambda}_0
\end{align*}
Now $\alpha(x,0,x\hat{\lambda}_0,\omega) = \alpha(0,0,0,\pm 1): = \alpha$, which is a constant in the integrals. Here we used the fact that $\alpha(x,y,\lambda,\omega)$ is a quadratic form in the fibre variable $\omega$, hence changing the sign of $\omega$ does not change its value. 
We choose $\chi$ to be a Gaussian density (with respect to $\hat{\lambda}_0$) first, then we use approximation argument to obtain one that has compact support in $\hat{\lambda}_0$. We choose $\chi(0,\hat{\lambda}_0,0)= e^{-\frac{F\hat{\lambda}_0^2}{2\alpha}}$, then $a_{1,F}(0,\xi_{sc,0})$ can be rewritten as
\begin{align*}
& \int  (\int e^{-F\hat{\lambda}_0 \hat{t}_0-\frac{F\hat{\lambda}_0^2}{2\alpha} + i\xi_{sc,0} \hat{\lambda}_0\hat{t}_0} d\hat{\lambda}_0 )e^{-F\alpha \hat{t}_0^2 +i\xi_{sc,0} \alpha \hat{t}_0^2} d\hat{t}_0
\end{align*}
The integral in $\hat{\lambda}_0$ is a Fourier transform of Gaussian density, it is $\sqrt{\frac{2\pi \alpha}{F}} e^{\frac{\alpha F \hat{t}_0^2}{2}-i\xi_{sc,0} \alpha \hat{t}_0^2 -\frac{\alpha}{2F}\hat{t}_0^2\xi_{sc,0}^2}$, and we can also get a similar expression for the term with $\omega = -1$. Thus we need to compute:
\begin{align*}
\int e^{-\frac{\alpha}{2F}(F^2+\xi_{sc,0}^2) \hat{t}_0^2 } d\hat{t}_0,
\end{align*}
which is again a Gaussian type integral, and it equals to a constant multiple of $\sqrt{\frac{F}{\alpha}}(F^2+\xi_{sc,0}^2)^{-\frac{1}{2}}$, and the other term associated to $\omega=-1$ gives the same contribution. Finally, with a constant factor $\hat{C}$, we have (again using symmetry of $\varrho$):
\begin{align}  \label{symbol_boundary}
a_{1,F}(0,\xi_{sc,0}) & = \hat{C} \varrho((0,0,0,1),(0,0))(\sqrt{\frac{F}{\alpha}}(F^2+\xi_{sc,0}^2)^{-\frac{1}{2}}),
\end{align}
which proves ellipticity of the boundary principal symbol.

Now we amend the compact support issue. Let $\chi$ be a Gaussian as above, which generates an elliptic operator, then we pick a sequence $\chi_n \in \mathcal{C}_c^\infty(\mathbb{R})$ converges to $\chi$ in the Schwartz function space $\mathcal{S}(\mathbb{R})$. Then we can obtain the convergence of their partial Fourier transform
in $\hat{\lambda}_0$: $\hat{\chi_n}$ to $\hat{\chi}$ in Schwartz function space. This gives us the convergence of $X-$Forier transform in $\mathcal{S}(\mathbb{R})$. In particular, we obtain the convergence of $|\zeta|a_{1,F;n}(x,\xi_{sc,0})$, the symbol obtained from $\chi_n$, in the $\mathcal{C}^0$ topology, which is enough to derive an elliptic type estimate for $\chi_n$ with large enough $n$.
\end{proof}

\subsection{Ellipticity near the \texorpdfstring{$x=c$}{x=c} }
In this part we show the pseudodifferential property and the ellipticity of $a_{1,F}$ near $\Sigma_c$, which is the level set of $\tilde{x}$ that is tangent to $\partial X$.
The pseudodifferential property needs more justification compared with
the case near the artificial boundary due to the singularity of $\tau_{\pm}(x,0,\lambda,\omega)$ at $x=c$. 

First we show that the singularity of $\tau_\pm(x,0,\lambda,\omega)$ as $x \rightarrow c$ is of square root type. 
Recall $\rho_c=c-x$, we have following characterization of the escape times:
\begin{lmm} When $c$ is small, the escape times
from the transversal $\{y=0\}$ we choose has a square root
type singularity when $\rho_c \rightarrow 0$. Concretely, we have
\label{lmm_singularity_escapetime}
\begin{align*}
\tau_\pm(x,0,\lambda,\omega) = \sqrt{\rho_c} \hat{\tau}_\pm(x,\lambda,\omega)
\end{align*}
with $\hat{\tau}_{\pm}$ being a smooth function, when $\rho_c$ is small and $|\lambda| \leq C \rho_c$. In addition, $\hat{\tau}_\pm$ is bounded away from 0.
\end{lmm}

\begin{proof}
Recall (\ref{tildex}), we know that on the boundary we have
\begin{align}
\rho_c = -\tilde{x} = \epsilon |z-p|^2 = \epsilon y^2 + \epsilon \rho_c^2,
\end{align}
which gives
\begin{align}
y = \pm \epsilon^{-1/2}\rho_c^{1/2}(1-\rho_c)^{1/2}
\end{align}
With starting velocity $\omega=\pm 1$ on the $y-$direction, and the velocity on this direction (this component of fiber part on the tangent bundle) varies smoothly,
and the factorization as in the lemma follows.
\end{proof}

Return to the proof of the pseudodifferential property and the ellipticity.
We write $\rho_c(\gamma(t))$ for $\rho_c-$component of the
point $\gamma(t)$. When we write $\rho_c$ alone, we mean $\rho_c(\gamma(0))$,
i.e., this component of the starting point.

Similar to the case near $x=0$ but now with $\Phi_F(x)=\frac{F}{\rho_c}$, the Schwartz kernel is:
\begin{align*}
K_{A_{1,F}}(\rho_c,\rho_c') = \int & e^{\frac{F}{\rho_c} - \frac{F}{\rho_c(\gamma_{x,0,\lambda,\omega}(t))}} \rho_c^{-2}\chi(x,\frac{\lambda}{x},\frac{\lambda}{\rho_c})\delta(z'-\gamma_{x,0,\lambda,\omega}(t))  
\\ & \varrho((x,0,\lambda,\omega),\gamma_{x,0,\lambda,\omega}(t)) \chi_X(\gamma(t)) dt |d\nu| dy'
\\ = (2 \pi)^{-n} \int & e^{\frac{F}{\rho_c}  - \frac{F}{\rho_c(\gamma_{x,0,\lambda,\omega}(t))}}   \rho_c^{-2}\chi(x,\frac{\lambda}{x},\frac{\lambda}{\rho_c})  e^{-i \zeta' \cdot ((\rho_c',y')-\gamma_{x,0,\lambda,\omega}(t))} 
 \\ & \varrho((x,0,\lambda,\omega),\gamma_{x,0,\lambda,\omega}(t)) \chi_X(\gamma(t))  dt |d\nu|d\zeta'dy'.
\end{align*}
As before, using the escape time, we can remove $\chi_X$ in the expression above:
\begin{align*}
K_{A_{1,F}}(\rho_c,\rho_c') = & (2 \pi)^{-n} \int_{\nu,\zeta'} \int_{\tau_-(x,0,\lambda,\omega)}^{\tau_+(x,0,\lambda,\omega)} 
e^{\frac{F}{\rho_c} - \frac{F}{\rho_c(\gamma_{x,0,\lambda,\omega}(t))}}   \rho_c^{-2}\chi(x,\frac{\lambda}{x},\frac{\lambda}{\rho_c})
 \\& e^{-i \zeta' \cdot ((\rho_c',y')-\gamma_{x,0,\lambda,\omega}(t))} 
 \varrho((x,0,\lambda,\omega),\gamma_{x,0,\lambda,\omega}(t))  dt |d\nu|d\zeta' dy '.
\end{align*}
Then convert this to the symbol by inverse Fourier transform in $\rho_c'$ evaluated at $\xi$, and this combined with 
the $e^{-i\zeta' \cdot (\rho_c',y') }$ gives $\delta_0(\xi-\xi')$, while the term
$e^{i\eta' (y(\gamma(t))-y')}$ is left (use $\gamma(t)$ to represent $\gamma_{x,0,\lambda,\omega}(t)$ below),
we have
\begin{align*}
a_{1,F}(\rho_c,\xi_{sc,c}) =  &  (2 \pi)^n e^{-i\rho_c\xi } \mathcal{F}^{-1}_{\rho_c'\rightarrow \xi }K_{A_{1,F}}(\rho_c,\rho_c')  \\
  =  \int_{\lambda,\omega,\eta'} \int_{\tau_-(x,0,\lambda,\omega)}^{\tau_+(x,0,\lambda,\omega)}   & e^{\frac{F}{\rho_c} - \frac{F}{\rho_c(\gamma_{x,0,\lambda,\omega}(t))}} \rho_c^{-2}\chi(x,\frac{\lambda}{x},\frac{\lambda}{\rho_c})  e^{-i \rho_c {\xi} } e^{i {\xi} \rho_c(\gamma(t))} e^{i\eta'(y(\gamma(t))-y')} 
  \\ & \varrho((x,0,\lambda,\omega),\gamma_{x,0,\lambda,\omega}(t))  dt  d\lambda d\omega d\eta' dy'.
\end{align*}
The $\eta'-$integral above produces a $\delta(y(\gamma(t))-y')$ and we can
evalute the $y'-$integration leaving the expression unchanged,
and we substitute in $\xi_{sc,c}=\frac{\xi}{\rho_c^2}$:
\begin{align}
\begin{split}
a_{1,F}(x,\xi_{sc,c}) =  &  (2 \pi)^n e^{-ix\xi} \mathcal{F}^{-1}_{\rho_c'\rightarrow \xi}K_{A_{1,F}}(x,\rho_c')  \\
  =  \int_{\lambda,\omega} \int_{\tau_-(x,0,\lambda,\omega)}^{\tau_+(x,0,\lambda.\omega)} & e^{\frac{F}{\rho_c} - \frac{F}{\rho_c(\gamma_{x,0,\lambda,\omega}(t))}} \rho_c^{-2}\chi(x,\frac{\lambda}{x},\frac{\lambda}{\rho_c})  e^{-i \rho_c \frac{\xi_{sc,c}}{\rho^2_c(\gamma(0))} } e^{i \frac{\xi_{sc,c}}{\rho_c^2(\gamma(0))} \rho_c(\gamma(t))}
  \\ & \varrho((x,0,\lambda,\omega),\gamma_{x,0,\lambda,\omega}(t)) dt  d\lambda d\omega .
\end{split}
\end{align}
Recalling that
\begin{align*}
\rho_c(\gamma(t)) = c-x(\gamma(t)) = \rho_c(\gamma(0))-\lambda t - \alpha t^2 - t^3\Gamma^{(1)},
\end{align*}
the phase is
\begin{align*} 
 - \frac{\xi_{sc,c}}{\rho_c^2}(\lambda t+\alpha t^2 + t^3\Gamma^{(1)}) 
\end{align*}
Introducing rescaled variables:
\begin{align}
\label{rescale_lambda_t}
\hat{\lambda}_c  = \lambda/\rho_c(\gamma(0)), \hat{t}_c = t/\rho_c(\gamma(0)),
\end{align}
we can rewrite the phase as
\begin{align} 
\phi = -\xi_{sc,c}(\hat{\lambda}_c\hat{t}_c+\alpha \hat{t}_c^2+\rho_c \hat{t}_c^3\Gamma^{(1)});
\end{align}
and the exponent introduced by conjugation is written as
\begin{align*}
\frac{F}{\rho_c} - \frac{F}{\gamma^{(1)}_{x,y,\lambda,\omega}(t)} = & -F(\lambda t + \alpha t^2 + t^3 \Gamma^{(1)}(x,y,x\hat{\lambda}_c,\omega,x\hat{t}_c)) \\ 
& \times (\rho_c(\rho_c-\lambda t - \alpha t^2 - t^3 \Gamma^{(1)}(x,y,x\hat{\lambda}_c,\omega,x\hat{t}_c)))^{-1}   \\
=& -F(\hat{\lambda}_c\hat{t}_c +\alpha \hat{t}_c^2 +\hat{t}_c^3x \hat{\Gamma}^{(1)}(x,y,x\hat{\lambda}_c,\omega,x\hat{t}_c)),
\end{align*}
where $\hat{\Gamma}_c^{(i)}$ is introduced when we first express $\gamma^{(1)}_{x,y,\lambda,\omega}(t)$ by variables $t,\lambda$, and then invoke our change of variables, then collect the remaining terms, which is a smooth function of these normalized variables. So this amplitude is Schwartz in $\hat{t}_c$, hence we take a constant $\epsilon_t>0$ and
deal with regions $|\hat{t}_c| \geq \epsilon_t$ and $|\hat{t}_c|< \epsilon_t$ separately. In our later arguments, we will take $\epsilon_t$ small
to enforce $\hat{t}_c=0$ holds for critical points.

Next we discuss the potential singularity introduced by $\chi_X(\gamma(t))$. 
This factor is turned into the bound $\tau_\pm(x,0,\lambda,\omega)$ of the $t-$integral, and after rescaling it becomes $\frac{\tau_\pm(x,0,\lambda,\omega)}{\rho_c}=\rho_c^{-1/2}\hat{\tau}_{\pm}(x,\lambda,\omega)$, where we have used Lemma \ref{lmm_singularity_escapetime}.
Under repeated application of $\rho_c\partial_{\rho_c}$, the contribution given by differentiating the bound of intergation and differentiating the integrand is
\begin{align}
\begin{split}
& (\rho_c\partial_{\rho_c})^k( (e^{-F(\hat{\lambda}_c \hat{t}_c + \alpha \hat{t}_c^2 + \hat{t}_c^3\rho_c \hat{\Gamma}^{(1)}(x,0,\rho_c\hat{\lambda}_c,\omega,\rho_c\hat{t}_c) )} \chi(\hat{\lambda}_c) \varrho((x,0,\rho_c\hat{\lambda}_c,\omega),\gamma_{x,0,\rho_c\hat{\lambda}_c,\omega}(\rho_c\hat{t}_c)) \\
& e^{i(\xi(\hat{\lambda}_c \hat{t}_c + \alpha \hat{t}_c^2 + \hat{t}_c^3 \rho_c \hat{\Gamma}^{(1)}(x,0,\rho_c\hat{\lambda}_c,\omega,\rho_c\hat{t}_c)))})
|_{\hat{t}_c=\rho_c^{-1/2}\hat{\tau}_{\pm}(x,\lambda,\omega)} 
\\& (\rho_c^{1/2}\partial_{\rho_c}\hat{\tau}_\pm(x,\lambda,\omega)-\frac{1}{2}\rho_c\hat{\tau}_\pm(x,\lambda,\omega))).
\end{split}
\label{derivative,symbol}
\end{align}
By the smoothness of $\Gamma^{(1)}$, it is uniformly bounded on the region we concern.
As we mentioned, we can make the uniform bound $T_g$ for $|t|$ small by shrinking $O_p$ so that the coefficient of $\hat{t}_c^2-$term in the exponent satisfies
\begin{align}
 \min (\alpha(x,y,\lambda,\omega)+t\hat{\Gamma}^{(1)}(x,0,\lambda,\omega,t)) = \mk{l}>0.
\end{align}
As a result, (\ref{derivative,symbol}) is bounded by a polynomial of $\rho_c^{-1/2}$
(with coefficient being smooth functions of other variables)
times $e^{-F(\mk{l}\rho_c^{-1}\hat{\tau}_{\pm}^2+\hat{\lambda}_c\rho_c^{-1/2}\hat{\tau}_\pm)}$, which is smooth as $\rho_c \rightarrow 0$. 
Thus the potential singularity carried by $\chi_X$ does not affect the 
proof of the pseudodifferential property in \cite[Section~3.5]{uhlmann2016inverse} or \cite[Section~5]{stefanov2004stability}.

As in the previous case, before considering the critical points of the phase for small $\rho_c>0$, we first consider the critical points of the phase at $\rho_c=0$. 
When $\rho_c=0$, the phase becomes
$$
-\xi_{sc,c} (\hat{\lambda}_c\hat{t}_c + \alpha \hat{t}_c^2).
$$
When $|\hat{t}_c| \geq \epsilon_t$, the derivative with respect to $\hat{\lambda}_c$ vanishes only when $\xi_{sc,c}=0$. Since our analysis is away from the zero section, the region $|\hat{t}_c| \geq \epsilon_t$ gives rapidly decaying contribution. 

The case $\rho_c>0$ can be dealt with the same method, but with more complicated computation. Notice that, $\alpha,\hat{\Gamma}^{(i)}$ take $\lambda = \rho_c \hat{\lambda}_c, t = \rho_c \hat{t}_c$ as variables, and produces an extra $\rho_c$ factor when we take partial derivatives with respect to $\hat{\lambda}_c,\hat{t}_c$. Concretely, the derivative with respect to $\hat{\lambda}_c$ is:
\begin{align*}
\frac{\partial \phi}{\partial \hat{\lambda}_c} &= -\xi_{sc,c} \hat{t}_c(1+x\hat{t}_c\partial_\lambda \alpha + x^2 \hat{t}_c^2 \partial_\lambda \Gamma^{(1)})\\
&= -\xi_{sc,c} \hat{t}_c(1+t\partial_\lambda \alpha + t^2 \partial_\lambda \Gamma^{(1)}).
\end{align*}
Recall that $|t| \leq T_g$ and we can choose $T_g$ to be small by shrinking $O_p$.
Thus when $|\hat{t}_c| \geq \epsilon_t$,
$\xi_{sc,c}\hat{t}_c$ is non-zero and is going to dominate other terms, 
so $\frac{\partial \phi}{\partial \hat{\lambda}_c}$ can not vanish and there is no critical point in this case. Since the $\epsilon_t$ in arguments above is arbitrary, we know that the condition $\hat{t}_c = 0$ holds for any critical point including the $\rho_c \neq 0$ case.

Next we consider the region $|\hat{t}_c|<\epsilon_t$, whose closure is compact, and consequently we can apply the stationary phase lemma. The same as before, we consider the condition that the derivative with respect to $\hat{\lambda}_c$ and $\hat{t}_c$ vanish. First consider the $\rho_c=0$, in which case the expression can be significantly simplified:
$$
\xi_{sc,c} \hat{t}_c = 0, \quad \xi_{sc,c} \hat{\lambda}_c  = 0.
$$
We exclude $\xi_{sc,c}=0$ case since we are away from the zero section. Then we have the condition for critical points:
$$
\hat{t}_c = 0, \quad \xi_{sc,c} \hat{\lambda}_c  = 0.
$$
 The second condition can be derived if we notice that (for general small $\rho_c$):
$$
\frac{\partial \phi}{\partial \hat{t}_c} = -\xi_{sc,c}\hat{\lambda}_c + O(\hat{t}_c),
$$
where the $O(\hat{t}_c)$ term vanishes when $\hat{t}_c=0$, and can be computed explicitly:
$$
-(2\xi_{sc,c} \alpha \hat{t}_c + 3x \xi_{sc,c} \Gamma^{(1)} \hat{t}_c^2 + \xi_{sc,c} x^2 \partial_t\Gamma^{(1)} \hat{t}_c^3).
$$
So those two conditions for stationary points extends to the $\rho_c \neq 0$ case.
In order to apply those conditions of critical points of the phase, we first rewrite (\ref{eq3}) as:
\begin{align}
\begin{split}
a_{1,F}(z,\xi_{sc,c})  &  = \int e^{-F(\hat{\lambda}_c \hat{t}_c + \alpha \hat{t}_c^2 + \hat{t}_c^3\rho_c \hat{\Gamma}^{(1)}(x,0,\rho_c\hat{\lambda}_c,\omega,\rho_c\hat{t}_c) )} \chi(\hat{\lambda}_c) \varrho((x,0,\rho_c\hat{\lambda}_c,\omega),\gamma_{x,0,\rho_c\hat{\lambda}_c,\omega}(\rho_c\hat{t}_c)) \\
&e^{-i(\xi_{sc,c}(\hat{\lambda}_c \hat{t}_c + \alpha \hat{t}_c^2 + \hat{t}_c^3 \rho_c \hat{\Gamma}^{(1)}(x,0,\rho_c\hat{\lambda}_c,\omega,\rho_c\hat{t}_c)))} d\hat{t}_c d\hat{\lambda}_cd\omega,
\end{split}
\label{aF2_c}
\end{align}
where integrating over $\omega$ is just summing two terms at $\pm 1$. 
By stationary phase lemma with a non-degenerate critical point, the leading contribution comes from the critical points of the phase $\{\hat{t}_c=0, \xi_{sc,c} \hat{\lambda}_c = 0 \}$. 
The $(\hat{t}_c,\hat{\lambda}_c)$-Hessian of the phase at the critical points is:
$$
-\begin{pmatrix} 2 \alpha \xi_{sc,c} & \xi_{sc,c} \\ \xi_{sc,c} & 0  \end{pmatrix},
$$
which has determinant $\xi_{sc,c}^2$. So the asymptotic behaviour of the integral as $|\xi_{sc,c}| \rightarrow \infty$ is the same as (up to a non-zero constant factor, and use the symmetry of $\varrho$ with respect to the vector fiber part)
\begin{align}
\label{symbol_main_term_c}
|\xi_{sc,c}|^{-1}\chi(0) \varrho((x,0,\lambda,\pm 1),(x,0)) .
\end{align}
Choosing $\chi$ such that $\chi \equiv 1$ near 0, we get a $-1$ order elliptic estimate near fiber infinity for both $\rho_c>0$ and $\rho_c=0$.

Next we turn to show boundary part of the principal symbol of $A_{1,F}$ is also elliptic (when the fiber variables are finite). 
Evaluating (\ref{aF2_c}) at $\rho_c=0$ (i.e., $x=c$), since $\omega = \pm 1$, the boundary principal symbol of $A_{1,F}$ is
\begin{align*}
a_{1,F}(0,y,\zeta) & = \int e^{-F(\hat{\lambda}_c \hat{t}_c + \alpha \hat{t}_c^2)} \chi(\hat{\lambda}_c) e^{-i \xi_{sc,c}(\hat{\lambda}_c \hat{t}_c + \alpha \hat{t}_c^2 ) }  \varrho((c,0,0,\omega),\gamma_{c,0,0,\omega}(0))   d\hat{t}_c  d\hat{\lambda}_cd\omega\\
& = \varrho((c,0,0,1),(0,y))\int e^{-F(\hat{\lambda}_c \hat{t}_c + \alpha \hat{t}_c^2)} \chi(\hat{\lambda}_c) e^{-i \xi_{sc,c}(\hat{\lambda}_c \hat{t}_c + \alpha \hat{t}_c^2 ) } d\hat{t}_c d\hat{\lambda}_c  \\
&+ \varrho((c,0,0,-1),(0,y)) \int e^{-F(\hat{\lambda}_c \hat{t}_c + \alpha \hat{t}_c^2)} \chi(\hat{\lambda}_c) e^{-i \xi_{sc,c}(\hat{\lambda}_c \hat{t}_c + \alpha \hat{t}_c^2 ) } d\hat{t}_c d\hat{\lambda}_c
\end{align*}
Now $\alpha(c,0,\rho_c\hat{\lambda}_c,\omega) = \alpha(c,0,0,\pm 1): = \alpha$, which is a constant in the integrals. Here we used the fact that $\alpha(c,0,0,\omega)$ is a quadratic form in the fibre variable $\omega$, hence changing the sign of $\omega$ does not change its value. 
We choose $\chi$ to be a Gaussian density (with respect to $\hat{\lambda}_c$) first, then we use approximation argument to obtain one that has compact support in $\hat{\lambda}_c$. We choose $\chi(x,\frac{\lambda}{x},\frac{\lambda}{\rho_c})= e^{-\frac{F(\lambda/\rho_c)^2}{2\alpha(0)}}$ when $c-x$ is small, then we have:
\begin{align*}
  & \int e^{-F(\hat{\lambda}_c \hat{t}_c + \alpha \hat{t}_c^2)} \chi(x,\hat{\lambda}_0,\hat{\lambda}_c) e^{-i \xi_{sc,c}(\hat{\lambda}_c \hat{t}_c + \alpha \hat{t}_c^2 ) } d\hat{t}_c d\hat{\lambda}_c \\
= & \int  (\int e^{-F\hat{\lambda}_c \hat{t}_c-\frac{F\hat{\lambda}_c^2}{2\alpha} - i\xi_{sc,c} \hat{\lambda}_c\hat{t}_c} d\hat{\lambda}_c )e^{-F\alpha \hat{t}_c^2 - i\xi_{sc,c} \alpha \hat{t}_c^2} d\hat{t}_c
\end{align*}
The integral in $\hat{\lambda}_c$ is a Fourier transform of Gaussian density, it is $\sqrt{\frac{2\pi \alpha}{F}} e^{\frac{\alpha F \hat{t}_c^2}{2}+i\xi_{sc,c} \alpha \hat{t}_c^2 -\frac{\alpha}{2F}\hat{t}_c^2\xi_{sc,c}^2}$, and we can also get a similar expression for the term with $\omega = -1$. Thus we need to compute:
\begin{align*}
\int e^{-\frac{\alpha}{2F}(F^2+\xi_{sc,c}^2) \hat{t}_c^2} d\hat{t}_c,
\end{align*}
which is again a Gaussian type integral, and it equals to a constant multiple of $\sqrt{\frac{F}{\alpha}}(F^2+\xi_{sc,c}^2)^{-\frac{1}{2}}$. The other term with $\omega=-1$ gives the same contribution. Finally, with a constant factor $\hat{C}$, we have (again using symmetry of $\varrho$):
\begin{align}  \label{symbol_boundary_c}
a_{1,F}(0,y,\zeta) & = \hat{C}\varrho((c,0,0,1),(c,0))(\sqrt{\frac{F}{\alpha}}(F^2+\xi_{sc,c}^2)^{-\frac{1}{2}}),
\end{align}
which proves ellipticity of the boundary principal symbol.

Now we amend the compact support issue. Let $\chi$ be a Gaussian as above, which generates an elliptic operator, then we pick a sequence $\chi_n \in \mathcal{C}_c^\infty(\mathbb{R})$ converges to$\chi$ in the Schwartz function space $\mathcal{S}(\mathbb{R})$. Then we can obtain the convergence of $\hat{\chi_n}$ to $\hat{\chi}$ in Schwartz function space. 
In particular, we obtain the convergence of $ |\xi_{sc,c}| a_{n,F}(x,\xi_{sc,c})$, the symbol obtained from $\chi_n$, in the $\mathcal{C}^0$ topology, which is enough to derive an elliptic type estimate for $\chi_n$ with large enough $n$.

\section{The proof of the main theorem}
\label{main_proof}
Fix $c_0$ small and apply results in previous sections to $\Omega_{c_0}$, estimates above are uniform with respect to $c \in (0, c_0]$. 
We let $c$ vary and take $\mk{f} \in \mathcal{F}_{\tilde{x}}(X)$ such that on the region $\Omega_c$ we have $x \leq \mk{f}$ and $\mk{f} = 0$ when $x=0$. 

Denoting the $A_{1,F}$ in the previous section, which constructed for a fixed $c$,
by $A_c$, then by the ellipticity of $A_c$ as scattering pseudodifferential
operator near both boundaries we have its parametrix $G_c$ such that 
\begin{align*}
G_cA_c = \Id+ E_{c}, \, E_{c} \in \Psi^{-\infty,-\infty,-\infty}_{sc,sc}([0,c]).
\end{align*}
This parametrix is constructed by choosing 
\begin{align*}
G_c = \chi_1(x/c)G_{0,c} + (1-\chi_1(x/c))G_{c,c},
\end{align*}
where $\chi_1 \in C_c^\infty(\R)$ is supported in $[-1/2,1/2]$ and equals to $1$
on $[-1/4,1/4]$; $G_{0,c}$ is the parametrix of $A_{1,F}$ as a scattering pseudodifferential operator taking $x=0$ as the boundary, and $G_{c,c}$ is the parametrix of $A_{1,F}$ as a scattering pseudodifferential operator taking $x=c$ as the boundary.


We consider the Schwartz kernel $K_{E_c}$ of $E_c$, which satisfies 
\begin{align*}|x^{-N}(c-x)^{-N}x'^{-N}(c-x')^{-N}K_{E_c}| \leq C_N \text{ on } [0,c]_x\times [0,c]_{x'}.
\end{align*}
$C_N$ and $C_N'$ below can be chosen to be independent of $c$ since our construction is uniform for small $c$ (in particular, the constant in elliptic estimate and the parametrix construction are uniform for small $c$).

This means
\begin{align*}
 |K_{E_c}| \leq C'_N\mk{f}(c)^{4N}x^{2}x'^{2}
\end{align*}
  for all $N$. The $x^2,(x')^2$ factors are introduced to `cancel' the scattering density. Then we apply Schur's lemma on the integral operator bound (together with the aforementioned  $N-th$ power estimate) to conclude that 
\begin{align*}
 ||E_c||_{L^2_{sc,sc}([0,c]) \rightarrow L^2_{sc,sc}([0,c])} \leq C''_N\mk{f}(c)^{4N},
\end{align*}
where $L^2_{sc,sc}$ is the function space with $m=l_1=l_2=0$ in Definition \ref{defn_function}.
In particular, we can take $c$ so that this norm $<1$. 
This guarantees that $\phi_cG_cA_c \phi_c =   \Id + \phi_cE_{c}\phi_c$ is invertible,
where $\phi_c(x)$ is a function supported in $[0,c]$ and is 1 on a smaller compact set $K_c$. Since $K_c$ is arbitrary,
for functions (in $\mathcal{F}_{\tilde{x}}$) supported on $\{\tilde{x} \geq -c\}$, $A_c$ is injective. 
We have
$$
||v||_{H^{s,r,r}_{sc,sc}([0,c])}  \leq C ||A_cv||_{H^{s+1,r,r}_{sc,sc}([0,c])}.
$$
If we recover this expression to $L\circ I_\varrho$, this is (with $f = Ee^{\Phi_F(x)}v$):
$$
||Rf||_{ e^{\Phi_F(x)} H^{s,r,r}_{sc,sc}([0,c]) }  \leq C ||R \circ L\circ I_\varrho f||_{ e^{\Phi_F(x)} H^{s+1,r,r}_{sc}([0,c])}.
$$

Recall our inclusion relationships for polynomially weighted Sobolev spaces, we can get rid of the $r-$indices with the cost of increasing the exponential power of left hand side to $e^{\Phi_{(F+\delta)}(x)}$ with $\delta>0$. That is:
 \begin{align}   \label{main_est_2}
||Rf||_{ e^{\Phi_{(F+\delta)}(x)} H^{s}_{sc,sc}([0,c])}  \leq C ||R \circ L\circ I_\varrho f||_{e^{\Phi_F(x)} H^{s+1}_{sc,sc}([0,c])}.
\end{align}
Finally we consider the boundedness of operators involved. We consider the decomposition 
$$A = R \circ L \circ I_\varrho \circ E.$$
By definition and the compactness of $X \cap \{ x \geq 0\}$ (hence finite length of each level
set of $\tilde{x}$), the extension operaotr $E$ extending a function on $[0,c]$ to a function on $K_c$ by assigning constant value on each level set of $\tilde{x}$ is bounded from sc-sc Sobolev spaces on $[0,c]$ to
the sc-sc Sobolev space on $\tilde{X} \cap \{0 \leq x \leq c\}$ with the same regularity and decay orders. Here the sc-sc Sobolev spaces on $\tilde{X} \cap \{0 \leq x \leq c\}$ take $\{x=0\}$ and $\{x=c\}$ (which are curves in $\tilde{X}$ instead of points in $[0,c]$) as boundaries. The construction is the same as in Section \ref{sec: combined psiDO}.

Then we show that $L$ is bounded. In order to prove this, we decompose $L$ into $L=M_2 \circ \Pi \circ M_1$, with $M_2,\Pi,M_1$ being 
\begin{align*}
 M_1:& H^s([0,c]_x \times \mathbb{R}_y \times \mathbb{R}_\lambda \times \{\pm1\}_\omega) \rightarrow H^s([0,c]_x \times \mathbb{R}_y \times \mathbb{R}_\lambda \times \{\pm1\}_\omega), \\
   &(M_1u)(x,y,\lambda,\omega) = x^s\rho_c^s \chi(x,\frac{\lambda}{x},\frac{\lambda}{\rho_c})u(x,y,\lambda,\omega),\\
 \Pi: & H^s([0,c]_x \times \mathbb{R}_y \times \mathbb{R}_\lambda) \rightarrow H^s([0,c]_x \times \mathbb{R}_y),
 \\& (\Pi u)(x,y) = \int_\mathbb{R} u(x,y,\lambda,1)d\lambda,\\
 M_2:& H^s([0,c]_x \times \mathbb{R}_y) \rightarrow x^{-(s+1)}\rho_c^{-(s+1)}H^s([0,c]_x \times \mathbb{R}_y), 
 \\& (M_2f)(x,y)=x^{-(s+1)}\rho_c^{-(s+1)}f(x,y).\\
\end{align*}
Consider the boundedness of $M_1$ when $s \in \N$ first. The general case follows from interpolation. Consider derivatives of $x^s\rho_c^s \chi(x,\frac{\lambda}{x},\frac{\lambda}{\rho_c})u(x,y,\lambda,\omega)$ up to order $s$. Each order of differentiation on $\chi$ gives an $x^{-1}$ or $\rho_c^{-1}$ factor, which is cancelled by $x^s\rho_c^s$ and the remaining part belongs to $L^2$ by smoothness of $\chi$ and $u\in H^s$. 
$M_2$ is bounded by the definition of the space on the right hand side. The operator $\Pi$ is a pushforward map, integrating over $|\lambda| \leq C |x|, |\lambda| \leq C|\rho_c|$ (notice the support condition after we apply $M_1$), hence bounded.

On the other hand, $I_\varrho$ itself is a bounded operator. This comes from the decomposition $I_\varrho = \tilde{\Pi} \circ \Phi^*$, where $\Phi$ is the geodesic coordinate representation $\Phi(z,\nu,t)=\gamma_{z,\nu}(t)$ and $\tilde{\Pi}$ is integrating against $t$, which is bounded as a pushforward map. Because the initial vector always has length 1 on the tangent component, the travel time is uniformly bounded. $\Phi$ is one component of $\Gamma$ and the later is a diffeomorphism when we shrink the region. So $\Phi$ has surjective differential, hence the pull back is bounded. Consequently $I_\varrho$ is bounded.

Next we discuss the boundedness of $R$: $H_{sc,sc}^{s+\frac{1}{2},r,r}( O_p ) \rightarrow H_{sc,sc}^{s,r,r}([0,c])$. Here the sc-sc Sobolev space on $O_p$ is constructed using $\{\tilde{x}=c\} \cap O_p$ and $\{\tilde{x}=0\} \cap O_p$ as boundaries.
Then we have:
\begin{lmm} \label{lemma: trace}
For $s > 0$, the restriction map
\begin{align*}
R: e^{\Phi_F(x)}H_{sc,sc}^{s+\frac{1}{2}}( O_p ) \rightarrow e^{\Phi_F(x)}H_{sc,sc}^{s}([0,c])
\end{align*}
is bounded.
\end{lmm}
\begin{proof}
We only need to prove boundedness when the function is supported near $\{x=0\}$, and the boundedness when the function is supported near $\{x=c\}$ can be proved in the same way and these two bounedness combined will prove the result. 

We prove the case without the weight first.
Let $v(x,y) \in H_{sc}^{s+\frac{1}{2}}(O_p)$, and denote its Fourier transform (after locally reduce to $\R^n$, which is why we need to assume it to be supported near $\{x=0\}$ only) by $\hat{v}(\xi,\eta)$. Since 
\begin{align*}
Rv(x) = & (2\pi)^{-2}\int e^{i(x,0) \cdot (\xi,\eta)}\hat{v}(\xi,\eta)d\xi d\eta
\\ = & (2\pi)^{-2}\int e^{ix \xi}(\int \hat{v}(\xi,\eta) d\eta)d\xi,
\end{align*}
we know
\begin{align*}
\mathcal{F}_{x \rightarrow \xi}(Rv)(\xi)
= \int \hat{v}(\xi,\eta)d\eta.
\end{align*}
Applying the Cauchy-Schwartz inequality, we have
\begin{align*}
|\mathcal{F}_{x \rightarrow \xi}(Rv)(\xi)|^2
\leq  (\int \la (\xi,\eta) \ra^{-2s-1}d\eta)
(\int \la (\xi,\eta) \ra^{2s+1} |\hat{v}(\xi,\eta)|^2 d\eta ).
\end{align*}
Suppose we introduce $\tilde{\eta} = \frac{\eta}{\la \xi \ra}$, then
\begin{align*}
  & \int \la (\xi,\eta) \ra^{-2s-1}d\eta 
  = \la \xi \ra^{-2s}  \int \la \tilde{\eta} \ra^{-2s-1} d\tilde{\eta}
\lesssim  \la \xi \ra^{-2s}.
\end{align*}
Combining with previous estimates, we have
\begin{align*}
\la \xi \ra^{2s} |\mathcal{F}_{x \rightarrow \xi}(Rv)(\xi)|^2 \lesssim  \int \la (\xi,\eta) \ra^{2s+1} |\hat{v}(\xi,\eta)|^2 d\eta,
\end{align*}
and further integrate with respect to $\xi$ gives
\begin{align*}
||Rv||_{H^{s}_{sc}([0,c))}
\lesssim ||v||_{H^{s}_{sc}(O_p)},
\end{align*}
where $v$ is supported away from $\{x=c\}$ and the
scattering Sobolev space on the right hand side is taking $x=0$ as the boundary surface.
Replacing $v$ by $e^{-\Phi_F(x)}v$ gives
\begin{align*}
||Rv||_{e^{\Phi_F(x)}H^{s}_{sc}([0,c))}
\lesssim ||v||_{e^{\Phi_F(x)}H^{s}_{sc}(O_p)},
\end{align*}

Using the same arguments, we can prove 
\begin{align*}
||Rv||_{e^{\Phi_F(x)}H^{s}_{sc}((0,c])}
\lesssim ||v||_{e^{\Phi_F(x)}H^{s}_{sc}(O_p)},
\end{align*}
where $v$ is supported away from $\{x=0\}$ and the
scattering Sobolev space on the right hand side is taking $x=c$ as the boundary surface. 
For general $v$, multiplying by a smooth cut-off functions,
$\chi_2(x),1-\chi_2(x)$ with $\chi_2$ supported in $[-c/2,c/2]$ and combining previous two estimates
 proves the lemma.
\end{proof}
Then we take $g=L\circ I_\varrho f$, combining Lemma \ref{lemma: trace} (notice that $s$ in Lemma \ref{lemma: trace} is $s+1>0$ in the main Theorem) with the boundedness of $L$ and (\ref{main_est_2}) gives us an estimate
\begin{align*}
 ||Rf||_{e^{\frac{F}{x}}H^{s}_{sc}([0,c])} \leq C_1 ||I_\varrho f||_{H^{s+\frac{3}{2}}(PSX|_{\bar{M}_c})},
\end{align*}
where we require $f$ to have supported in $K_c$, and used the fact $\R_\lambda \times \{\pm1\}_\omega$ parametrizes $\mathbb{S}^1$ apart from two poles, and this completes the proof.

\bibliographystyle{plain}
\bibliography{bib_inverse}

\end{document}